\documentclass[12pt,twoside]{article}
\usepackage{amsmath, amsthm, amscd, amsfonts, amssymb, graphicx, color}
\usepackage[bookmarksnumbered, colorlinks]{hyperref} \usepackage{float}
\usepackage{lipsum}
\usepackage{afterpage}
\usepackage[labelfont=bf]{caption}
\usepackage[nottoc,notlof,notlot]{tocbibind} 

\usepackage{lipsum}
\usepackage{fancyhdr}
\newtheorem{theorem}{Theorem}[section]
\newtheorem{lemma}[theorem]{Lemma}
\newtheorem{proposition}[theorem]{Proposition}

\theoremstyle{definition}
\newtheorem{definition}[theorem]{Definition}
\newtheorem{example}[theorem]{Example}

\newtheorem{remark}[theorem]{Remark}
\newtheorem{question}[theorem]{Question}
\theoremstyle{remark}
\renewenvironment{proof}{{\bfseries \noindent Proof.}}{~~~~$\square$}
\makeatletter
\def\th@newremark{\th@remark\thm@headfont{\bfseries}}
\makeatletter

\begin{document}
\baselineskip 9mm

{

\vspace*{9mm}

\begin{center}

{\Large \bf 
ON MATCHING PROPERTY FOR GROUPS AND FIELD EXTENSIONS}

\let\thefootnote\relax\footnote{\scriptsize 2010 \textit{Mathematics Subject Classifcation.} Primary: 05D15; Secondary: 11B75, 20D60, 20F99,12F99.}
\let\thefootnote\relax

{\bf MOHSEN ALIABADI, MAJID HADIAN, AND AMIR JAFARI}
\vspace{2mm}

\end{center}

\vspace{4mm}

{\noindent \bf Abstract.}  In this paper we prove a sufficient condition for the existence of matchings in
arbitrary groups and its linear analogue, which lead to some generalizations of the existing
results in the theory of matchings in groups and central extensions of division rings. We
introduce the notion of relative matchings between arrays of elements in groups and use
this notion to study the behavior of matchable sets under group homomorphisms. We
also present infinite families of prime numbers $p$ such that $\mathbb{Z}/p\mathbb{Z}$ does not have the acyclic
matching property. Finally, we introduce the linear version of acyclic matching property
and show that purely transcendental eld extensions satisfy this property.

\tableofcontents

\section*{Introduction }
\addcontentsline{toc}{section}{Introduction }
\label{intro} 
Let $B$ be a finite subset of the group $\mathbb{Z}^n$ which does not contain the neutral element. For any subset $A$ in $\mathbb{Z}^n$ with the same cardinality as $B$, a matching from $A$ to $B$ is defined to be a bijection $f : A \to B$ such that for any $a\in A$ we have $a + f(a) \notin A$. For any matching $f$ as above, the associated multiplicity function $m_f : \mathbb{Z}^n \to \mathbb{Z}_{\geq 0}$ is defined via the rule:
\[
\forall x\in \mathbb{Z}^n, m_f (x) =|\{ a \in A : a + f(a) = x\} |.
\]
A matching $f : A \to B$ is called acyclic if for any matching $g : A \to B, m_f = m_g$ implies
$f = g$. Now the question is, fixing finite subsets $A$ and $B$ in $\mathbb{Z}^n$ with same cardinality such that $0\notin B$, is there an acyclic matching from $A$ to $B$? This question and the related notions were studied in \cite{7} by Fan and Losonczy. Their motivation is the relation between acyclic matchings and an old problem concerning elimination of monomials in a generic homogenous form under a linear change of variables, which was studied by Wakeford in 1916 (see \cite{11}). More precisely, Fan and Losonczy in \cite{7,8} use the existence of acyclic matchings for subsets of $\mathbb{Z}^n$ in order to show that any small enough fixed set of monomials can be removed from a generic homogeneous form after a suitable linear change of variables.\\
Later, the notions of matchings and acyclic matchings were generalized and studied in the
context of arbitrary abelian and even non-abelian groups. Let $A$ and $B$ be two finite subsets of an arbitrary group $G$. A matching from $A$ to $B$ is a bijection $f : A\to B$ such that for all
$a \in A, af(a) \notin A$. Evidently, it is necessary for the existence of a matching from $A$ to $B$ that $|A| = |B|$ and that $e \notin B$ (here $e$ denotes the neutral element of $G$). One says that a group $G$ has the matching property if these necessary conditions are sufficient as well. Moreover, the notions of multiplicity function associated to a matching and acyclic matchings are defined similar to the case $G =\mathbb{Z}^n$ mentioned above. $A$ group $G$ has the acyclic matching property if for any pair of subsets $A$ and $B$ in $G$ with $|A| =|B|$ and $e \notin B$, there is at least one acyclic matching from $A$ to $B$.\\
It is shown in \cite{7} that any free abelian group satisfies the acyclic matching property. As for
the matching property, Losonczy proves in \cite{8} that an abelian group satisfies the matching
property if and only if it is either torsion free or finite of prime order. This latter result of Losonczy has been generalized for arbitrary groups by Eliahou and Lecouvey (see \cite{5}). We would like to mention that, although all finite groups of prime order are known to satisfy the matching property, the classification of those prime numbers $p$ such that $\mathbb{Z}/p\mathbb{Z}$ has the acyclic matching property is unsolved.\\
In this paper, we prove a suficient condition for the existence of matchings in arbitrary groups, which leads to some generalizations of the above mentioned results concerning the matching property for groups. In particular, we prove a result which provides us with a systematic way for constructing matchings between subsets of groups which are not necessarily torsion free or of prime order. In order to deal with these groups, we will introduce the generalization of matchings between subsets to matchings between arrays of elements in groups relative to a normal subgroup and use this notion in studying the behavior of matchings under group homomorphisms. We also present infinite families of prime numbers p such that $\mathbb{Z}/p\mathbb{Z}$ fails to satisfy the acyclic matching property. On the other hand, we are not able to prove or disprove the existence of an infinite family of primes p such that $\mathbb{Z}/p\mathbb{Z}$ does satisfy
the acyclic matching property.\\
A related notion is that of a matching between subspaces of a central extension of division
rings. In \cite{6}, Eliahou and Lecouvey formulate some linear analogues of matchings in groups
and prove some similar results in the linear context. We also extend our results on matchings
in groups to the linear setting, which generalize some results of \cite{6}. We conclude by introducing the linear analogue of an acyclic matching and show that pure transcendental extensions have the linear acyclic matching property. For more results on matching property see \cite{2,3,4}.\\
\textbf{Organization of the paper}: In section 1 we prove a suficient condition for the existence of
matchings in groups, which generalizes some of the known results in the theory of matchings
in groups. Then we introduce the notion of a relative matching between two arrays of elements
in a group, which is a generalization of the usual notion of matching. We use this in order to
study the behavior of matchable sets under group homomorphisms. In section 2 we construct
infinite families of prime numbers $p$ such that the group $\mathbb{Z}/p\mathbb{Z}$ fails to satisfy the acyclic matching property. In section 3 we formulate and prove the linear version of the main result of section 1. Finally, in section 4 we introduce the linear analogue of the acyclic matching property and prove that any purely transcendental field extension satisfies this property. This result is the linear counterpart of the fact that free abelian groups satisfy the acyclic matching
property.
\section{A sufficient condition for the existence of matchings}\label{sec1}
Our goal in this section is to prove a suficient condition for the existence of matchings
in arbitrary groups, which generalizes some of the known results concerning the matching
property mentioned in the introduction. We also introduce the notion of a matching between
two arrays of elements in a group and use it for a systematic construction of matchings between subsets of groups which are not necessarily torsion free or of prime order.\\
The idea behind our first result is the following simple observation which shows that
existence of nontrivial proper finite subgroups is an obstruction for the matching property.
Assume that a group $G$ contains a nontrivial proper finite subgroup $H$. Let $A = xH$ be any
right coset of H and put $B = (H \setminus \{e\})\cup \{g\}$ for any element $g \in G\setminus H$. Then $A$ and $B$ are finite subsets of $G$ with $|A| = |B|$ and $e\notin B$, and there is no matching from $A$ to $B$. Indeed, if an element $xh \in A$ is matched to an element $h'\in B \cap H$, then $xhh' \in xH = A$ and this contradicts the definition of matching. Inspired by this example, we want to prove that if A does not contain any coset of any nontrivial proper finite subgroup of $G$, then there always exists a matching from A to any subset $B$ with $|A| = |B|$ and $e\notin B$. For this purpose, we need the following result of Olson.
\begin{theorem}\label{th1.1}
\cite[Theorem 2]{10}
Let $X$ and $Y$ be two nonempty finite subsets of a group $G$ and put $XY = \{xy : x \in X; y \in Y \}$. Then there exists a finite subgroup $H$ of $G$ and a nonempty subset $T$ of $XY$ such that
\[
|T|\geq |X| + |Y |-|H|,
\]
and either $HT = T$ or $TH = T$.\\
Now we can prove the following suficient condition for existence of matchings.
\end{theorem}
\begin{theorem}\label{th1.2}
Let $G$ be an arbitrary group and A be a finite subset of $G$ which does not contain any (left or right) coset of any proper nontrivial finite subgroup of $G$. Then for any finite subset $B$ of $G$ with $|A| = |B|$ and $e\notin B$, there is a matching from $A$ to $B$.
\end{theorem}
\begin{remark}\label{re1.3}
Note that if $G$ is torsion free or finite of prime order, then the hypothesis of the above theorem is automatically satisfied for any finite subset $A$, since $G$ has no nontrivial proper finite subgroups at all. Therefore, the above theorem, in particular, implies that torsion free groups and finite groups of prime order have the matching property.
\end{remark}
\begin{proof}
(of Theorem \ref{th1.2}) Assume by the way of contradiction that there is no matching from
$A$ to $B$. For any element $a \in A$ let $M_a := \{b \in B : ab \notin A\}$ be the subset of $B$ consisting of those elements that are matchable to a. Then, Hall's marriage Theorem implies that there is a subset $S$ of $A$ such that:
\[
|S| > |\cup _{s\in S} M_s|
\]
By taking the complement and noticing that $A$ and $B$ have the same cardinality, this would
imply that:
\[
|A|-|S| < |\cap _{s\in S} M^c_s|,
\]
where $M^c_s$ denotes the complement of $M_s$ in $B$. Now, if we put $W_S := (\cap _{s\in S}M^c_s ) \{e\}$, we have $SW_S\subset A $ and
\begin{equation}\label{eq1}
|A| < |S| + |W_S|- 1.
\end{equation}
Theorem \ref{th1.1}, applied to the subsets $S$ and $W_S$ in the group $G$, implies that there is a finite subgroup $H$ of $G$ and a nonempty finite subset $T$ of $SW_S$ such that:
\begin{equation}\label{eq2}
|SW_S|\geq |T|\geq |S|+ |W_S|-|H|,
\end{equation}
and
\begin{equation}\label{eq3}
\mathrm{either}\; HT = T \mathrm{or}\; TH = T.
\end{equation}
If $H = G$, then \eqref{eq3} implies that $T = G$ and thus $SW_S = G$. On the other hand, $SW_S$ is a subset of $A$ and hence we have to have $A = G$. But, since $B$ has the same cardinality as $A$, this would imply that $B = G$, which contradicts the assumption $e\notin B$. Therefore, $H$ has to be a finite proper subgroup of $G$. Now, since either $TH = T$ or $HT = T$ and $T$ is a subset of $A$, we conclude that $A$ contains a coset of $H$. On the other hand, by our assumption, A does not contain any coset of any nontrivial finite proper subgroup of $G$. Therefore, $H$ has to be the trivial subgroup. But then \eqref{eq2} implies that:
\[
|A|\geq |SW_S|\geq |S| + |W_S|-|H| = |S| + |W_S|- 1,
\]
which contradicts the inequality \eqref{eq1}.
\end{proof}\\
Along the same line of ideas used in the above argument, we can prove the following generalization of \cite[Proposition 3.4]{8}. Note that \cite[Proposition 3.4]{8} can only be applied to cyclic groups, while the following result works for an arbitrary abelian group.
\begin{proposition}\label{pr1.4}
Let $G$ be any abelian group and $A$ and $B$ be finite subsets of $G$ with the same cardinality. Assume further that for any element $b\in B, A$ does not contain any coset of the subgroup generated by $b$. Then there is a matching from $A$ to $B$.
\end{proposition}
\begin{proof}
For any subset $S$ of $A$, put $V_S := \{b \in B : S + b \subset A\}$ and $W_S := V_S \cup \{0\}$. By applying Kneser's Theorem (see \cite[Page 116, Theorem 4.3]{9}) to the subsets $S$ and $W_S$ in the group $G$, we know that there is a finite subgroup $H$ of $G$ such that:
\begin{equation}\label{eq4}
|S +W_S|\geq |S + H| + |W_S + H|-|H|,
\end{equation}
and
\begin{equation}\label{eq5}
S +W_S + H = S +W_S.
\end{equation}
Since $S +W_S$ is a subset of $A$, \eqref{eq5} implies that $A$ contains a coset of $H$. On the other hand, by our hypothesis, $A$ does not contain any coset of the subgroup generated by any element in $B$. Since a coset of $H$ contains a coset of the subgroup generated by b for any element $b$ in $H \cap B, H$ does not intersect with $B$, and thus $H\cap W_S =\{0\}$. Implementing this into \eqref{eq4}, we obtain that:
\begin{equation}\label{eq6}
|B| =|A|\geq |S +W_S|\geq |S + H| + |W_S|-1 \geq |S| + |V_S|.
\end{equation}
But then, since \eqref{eq6} is valid for any subset $S$ of $A$, Hall's marriage theorem implies that there is a matching from $A$ to $B$.
\end{proof}\\
Now, we want to outline some methods which can be combined with the above results and produce matchings between subsets of groups that are not necessarily torsion free or of prime
order. For this purpose, we will use the following generalization of the notion of matching.
\begin{definition}\label{de1.5}
Let $G$ be a group and $a = (a_1,\cdots , a_n)$ and $b = (b_1,\cdots , b_n)$ be two $n$-tuples
of elements of $G$ (note that repetitions are allowed in $n$-tuples). For a normal subgroup $N$
of $G$, a matching from $\mathfrak{a}$ to $\mathfrak{b}$ relative to $N$ is defined to be a permutation $\sigma \in S_n$ such that for any $1\leq i, j\leq n, a_ib_{\sigma (i)} \notin a_jN$. A matching relative to the trivial subgroup $N =\{e\}$ will be simply called a matching
\end{definition}
\begin{remark}\label{re1.6}
Note that if both $\mathfrak{a}$ and $\mathfrak{b}$ have $n$ distinct entries and we put $A =\{a_1,\cdots , a_n\}$ and $B = \{b_1,\cdots ,b_n\}$, then a matching from $\mathfrak{a}$ to $\mathfrak{b}$ relative to the trivial subgroup is nothing but a matching from $A$ to $B$ in the usual sense. Moreover, note that if $M$ and $N$ are two normal subgroups of $G$ with $M\subset N$, then any matching relative to $N$ is a priori a matching relative to $M$. In particular, any relative matching between two subsets gives also a usual matching between those subsets.
\end{remark}
The following remark, combined with Theorem \ref{th1.2}, provides us with a class of matchings
between arrays of elements of a group.
\begin{remark}\label{re1.7}
Let $G$ be a group and $\mathfrak{a}$ be an $n$-tuple of elements of $G$. Then the support of $\mathfrak{a}$, denoted by Supp($\mathfrak{a}$), is dened to be the subset of $G$ consisting of distinct entries of $\mathfrak{a}$. Now let $\mathfrak{a}$ and $\mathfrak{b}$ be two $n$-tuples of elements of $G$ such that there is a matching $f : \sup (\mathfrak{a}) \to \sup (\mathfrak{b})$ and that for every $a\in \sup (\mathfrak{a})$ the number of times that a appears as an entry of $\mathfrak{a}$ is the same as the number of times that $f(a)$ appears as an entry of $\mathfrak{b}$. Then $f$ can be lifted in an evident way to a matching $\tilde{f}:\mathfrak{a}\to \mathfrak{b}$.
\end{remark}
\begin{proposition}\label{pr1.8}
Let $\eta : G \to H$ be a group homomorphism and let  $\mathfrak{a}= (a_1,\cdots , a_n)$ and
$\mathfrak{b} = (b_1,\cdots , b_n)$ be two $n$-tuples of elements of $G$. Then there is a matching from $\eta (\mathfrak{a}) :=(\eta (a_1),\cdots , \eta (a_n))$ to $\eta (\mathfrak{b}) := (\eta (b_1),\cdots ,\eta (b_n))$ if and only if there is a matching from $\mathfrak{a}$ to $\mathfrak{b}$ relative to $\ker (\eta )$.
\end{proposition}
\begin{proof}
By definition, a matching from $\eta (\mathfrak{a})$ to $\eta (\mathfrak{b})$ is a permutation $\sigma \in S_n$ such that for any $1 \leq i, j\leq n, \eta (a_i)\eta (b_{\sigma (i)}) \neq \eta (a_j )$. But, since $\eta $ is a group homomorphism, $\eta (a_i)\eta (b_{\sigma (i)}) =\eta (a_ib_{\sigma (i)})$ is different from $\eta (a_j)$ if and only if $a_ib_{\sigma (i)} \notin a_j\ker (\eta )$. This implies that the same permutation that establishes a matching from $\eta (\mathfrak{a})$ to $\eta (\mathfrak{b})$ gives rise to a matching from $\mathfrak{a}$ to $\mathfrak{b} $  relative to $\ker (\eta )$, and vise versa
\end{proof}\\
\begin{example}\label{ex1.9}
As an application of the above proposition, let $1\to N \to G \stackrel{\eta }{\rightarrow }H \to 1$ be a short exact sequence of groups, i.e. $\eta  : G \to H$ is a group epimorphism with kernel $N$. Then, for any pair $\mathfrak{a}$ and $\mathfrak{b}$ of $n$-tuples of elements of $G$, a matching from $\eta (\mathfrak{a})$ to $\eta (\mathfrak{b})$ leads to a matching from $\mathfrak{a}$ to $\mathfrak{b}$ relative to $N$.\\
In particular, let $G =\prod _iG_i$  be a group and let $p_i : G \to G_i$ be the corresponding projection. For any ordered subset $A =\{a_1,\cdots , a_n\}$ of elements of $G$, let $A_i =(p_i(a_1),\cdots , p_i(a_n))$ be the $n$-tuple formed of the $i$-th components of elements of $A$ with respect to the given decomposition $G =\prod _iG_i$. Then, for any two subsets $A$ and $B$ in $G$, any matching from $A_i$ to $B_i$ in $G_i$ can be lifted to a matching from $A$ to $B$ (even to a matching relative to $\prod _{j\neq i} G_j$ ).\\
This, together with Theorem \ref{th1.2}, gives us a tool for constructing matchings between subsets of groups that are not necessarily torsion free or of prime order.
\end{example}
\section{Acyclic matching property for finite groups of prime order}\label{sec2}
It is shown in \cite{5} and \cite{8} (and it follows from Theorem \ref{th1.2}) that a group satisfies the matching property if and only if it is either torsion free or finite of prime order. But a similar classification for acyclic matching property is yet to be found. The fact that every abelian
torsion free group admits an order compatible with the group structure can be used to prove
that such groups satisfy the acyclic matching property (see \cite{8}). But characterizing nonabelian
torsion free groups and finite groups of prime order that satisfy the acyclic matching property remains an open problem. In this section, we give two infinite sequences of prime numbers p such that $\mathbb{Z}/p\mathbb{Z}$ does not satisfy the acyclic matching property. For this, we use
the following lemma.
\begin{lemma}\label{l2.1}
Let $G$ be an abelian group and $A$ be a finite subset of $G$ such that $|A|$ is odd and
$0\notin A$ (for abelian groups the neutral element is denoted by 0). Then every acyclic matching
$f$ from $A$ to itself has a fixed point, that is there exists an element $a \in A$ such that $f(a) = a$.
\end{lemma}
\begin{proof}
First, note that for any matching $f : A \to A$, the inverse bijection $f^{-1} : A\to A$
is a matching with the same multiplicity function as $f$ (note that since $G$ is assumed to
be abelian, $af^{-1}(a) = f^{-1}(a)a = (f^{-1}(a))f(f^{-1}(a)) \notin A)$. Therefore, if $f : A\to A$ is an acyclic matching, we have $f = f^{-1}$ and thus $f\circ f = Id_A$. This implies that $f$, viewed as a permutation of elements of $A$, has order two and hence can be decomposed as product of disjoint 2-cycles and 1-cycles. But since we assumed that $A$ has odd cardinality, there is at least one 1-cycle in the cycle decomposition of $f$, which means that f has at least one fixed
point. 
\end{proof}\\
Now we are ready to prove the following two propositions, each of which provides us with
an infinite family of prime numbers $p$ such that $\mathbb{Z}/p\mathbb{Z}$ does not satisfy the acyclic matching property.
\begin{proposition}\label{pro2.2}
Let $p$ be a prime number such that $p\equiv -1$ modulo 8. Then $\mathbb{Z}/p\mathbb{Z}$ does not satisfy the acyclic matching property
\end{proposition}
\begin{proof}
Let $(\mathbb{Z}/p\mathbb{Z})^*$ denote the set of nonzero elements of $\mathbb{Z}/p\mathbb{Z}$ and consider the subset $A =\{n^2 : n \in  (\mathbb{Z}/p\mathbb{Z})^*\}$ of nonzero squares modulo $p$. We claim that there is no acyclic matching from $A$ to $A$. First, note that $|A| =\dfrac{p-1}{2}$ , which is an odd number since $p\equiv -1$ modulo 8. Therefore, Lemma \ref{l2.1} implies that any acyclic matching $f$ from $A$ to $A$ has to have a fixed point. But if $f(n^2) = n^2$ for some $n \in (\mathbb{Z}/p\mathbb{Z})^*$, by definition of matching, we should have $2n^2 \notin A$. This would imply that 2 is not a square modulo $p$, which contradicts our assumption $p\equiv -1$ modulo 8.
\end{proof}\\
\begin{proposition}\label{pro2.3}
Let $p$ be a prime number such that the order of 2 modulo $p$ is an odd number. Then $\mathbb{Z}/p\mathbb{Z}$ does not satisfy the acyclic matching property.
\end{proposition}
\begin{proof}
Let $p$ be a prime number such that the order $m$ of 2 modulo $p$ is an odd number and
consider the subset $A =\{2^i : 0\leq i \leq m -1\}$ of all powers of 2 modulo $p$. We claim that
there are no acyclic matchings from $A$ to $A$. By Lemma \ref{l2.1}, any acyclic matching $f : A \to  A$ has a fixed point. But if $f(2^i) = 2^i$, then $2^i +2^i = 2^{i+1} \in A$, which contradicts the definition of a matching.
\end{proof}\\
For a small prime $p$, one can check directly whether or not $\mathbb{Z}/p\mathbb{Z}$ satisfies the acyclic matching property. But it would be nice if one could answer the following:
\begin{question}\label{qe2.4}
Are there infinitely many primes p such that $\mathbb{Z}/p\mathbb{Z}$ satisfies the acyclic match-ing property?
\end{question}
\section{Linear Matchings in Central Extensions}\label{sec3}
In this section we formulate and prove the linear analogue of Theorem \ref{th1.2} proven in Section 1. Throughout this section we assume that $K\subset L$ is a central extension of division rings, that is $L$ is a division ring and $K$ is a subfield of the center of $L$. For any subset $S$ of $L$, the $K$-subspace of $L$ generated by $S$ will be denoted by $\left\langle S\right\rangle $. Furthermore, for any pair of subsets $A$ and $B$ of $L$, the Minkowski product $AB$ of these subsets is defined as:
\[
AB := \{ab \mid a \in A, b \in B\}
\]
Recall that Eliahou and Lecouvey have introduced the following notions for a central extension 
$K\subset L$ of division rings (see \cite{6}). Let $A$ and $B$ be $n$-dimensional $K$-subspaces of $L$ for some $n\geq 1$. Then an ordered basis $\mathcal{A} =\{a_1,\cdots , a_n\}$ of $A$ is said to be matched to an ordered basis $\mathcal{B} =\{b_1,\cdots , b_n\}$ of $B$ if
\begin{equation}\label{eq7}
\forall 1\leq i\leq n, \; a_i^{-1}A\cap B\subset \left\langle b_1,\cdots , \hat{b}_i,\cdots ,b_n\right\rangle ,
\end{equation}
where $\left\langle b_1,\cdots , \hat{b}_i,\cdots ,b_n\right\rangle $ is the vector space generated by $\{b_1,\cdots , b_n\} \setminus \{b_i\}$.The subspace $A$ is matched to the subspace $B$ if every basis of $A$ can be matched to a basis of $B$. Finally, the extension $L$ of $K$ has the linear matching property if for every $n\geq 1$ and any pair $A$ and $B$ of $n$-dimensional $K$-subspaces of $L$ with $1\notin B, A$ is matched to $B$ (it is shown in \cite[Lemma 2.3]{6} that $1\notin B$ is a necessary condition for the existence of a matching). One of the main results in \cite{6} is that a central extension $K\subset L$ has the linear matching property if and only if there are no nontrivial finite intermediate extensions $K\subset M \subset L$. We would like to mention that, although the statement of \cite[Theorem 2.6]{6} is slightly different and assumes that the extension is either purely transcendental or finite of prime degree, what they actually use in their proof is that there are no nontrivial finite intermediate extensions, which is a weaker condition (see also \cite{1}). In the following, as an analogy with Theorem \ref{th1.2}, we give a generalization of this result (see Theorem \ref{th3.3}). The main ingredient in our proof is the following linear version of Olson's theorem.
\begin{theorem}\label{th3.1}
\cite[Theorem 4.3]{6} Let $K\subset L$ be a central extension of division rings and let $A$
and $B$ be two nonzero finite dimensional $K$-subspaces of $L$. Then there exists a nonzero $K$-subspace $S$ of $\left\langle AB \right\rangle $ and a finite dimensional sub-division ring $M$ of $L$ such that the following hold:
\begin{itemize}
\item[(1)]
$K \subset M \subset L$,
\item[(2)]
$\dim (S) \geq \dim (A) +\dim (B) -\dim (M)$,
\item[(3)]
$MS = S$ or $SM = S$.
\end{itemize}
We will also use the following definition.
\end{theorem}
\begin{definition}\label{df3.2}
Let $K \subset L$ be a central extension of division rings and $M$ be a sub-division ring in $L$. Then a left (resp., right) linear translate of $M$ is a $K$-subspace of the form $lM$ (resp., $Ml$) for a nonzero element $l\in L$.
\end{definition}
Now we are ready to prove the following linear version of Theorem \ref{th1.2}.
\begin{theorem}\label{th3.3}
 Let $K\subset L$ be a central extension of division rings and let $A$ be an $n$-dimensional $K$-subspace of $L$ which does not contain any (left or right) linear translate of a nontrivial finite dimensional sub-division ring of $L$. Then $A$ is matched to any $n$-dimensional $K$-subspace $B$ of $L$ provided $1\notin B$.
 \end{theorem}
 \begin{remark}\label{3.4}
  Note that the above theorem generalizes sufficiency of the condition in \cite[Theorem 2.6]{6} and also it's refinement \cite[Theorem 5.3]{6}. 
  \end{remark}
\begin{proof}
(of Theorem \ref{th3.3}) Fix an ordered basis $\mathcal{A} = \{a_1,\cdots ,a_n\}$ for $A$ and assume by the way of contradiction that $\mathcal{A}$ can not be matched to any basis of $B$. For any subset $I$ in $\{1,\cdots , n\}$, put
\[
V_I :=\bigcap _{i\in I}(a^{-1}_i A \cap B) = \{x \in B : a_ix \in A \mathrm{for \; all}\; i\in I\}.
\]
Then, by the linear version of Hall's marriage theorem (see \cite[Proposition 3.1]{6} for example), there is a subset $I\subset \{1,\cdots , n\}$ such that:
\begin{equation}\label{eq8}
\dim (V_I ) > n - |I|.
\end{equation}
Note that $K\cap V_I =\{0\}$, as $V_I \subset B$ and $1\notin B$, and put $W_I := V_I \bigoplus K$. Then $A_IW_I\subset A$, where $A_I =\left\langle \{a_i\}_{i\in I}\right\rangle $. Applying Theorem \ref{th3.1} to the subspaces $A_I$ and $W_I$ , we conclude that there exists a nonzero $K$-subspace $S$ of $\left\langle A_IW_I\right\rangle $ and a finite dimensional sub-division ring $M$ of $L$ such that:
\[
\dim (S)\geq \dim (A_I ) +\dim (W_I ) -\dim (M)
\]
and that either $MS = S$ or $SM = S$. This would mean that $S$, and therefore $A$, contains a
linear translate of $M$. Hence, by our assumption on $A, M$ has to be the trivial sub-division
ring in $L$, i.e. $M = K$, and thus
\[
\dim (S)\geq \dim (A_I ) +\dim (W_I ) - 1 =|I| +\dim (V_I ) > n,
\]
where the last inequality follows from \eqref{eq8}. But this contradicts the fact that $S \subset \left\langle A_IW_I\right\rangle \subset A$ and $\dim(A) = n$.
\end{proof}
\section{ Linear Acyclic Matchings}\label{sec4}
In this section we introduce the linear version of the notion of an acyclic matching and  prove that every purely transcendental field extension satisfies the linear acyclic matching property. Recall that a matching $f : A\to B$ between two finite subsets of a group is called acyclic if for every matching $g : A \to B$ with $m_f = m_g$, one has $f = g$. A group $G$ satisfies the acyclic matching property if for every pair $A$ and $B$ of finite subsets of $G$ with $|A| = |B|$
and $e \notin B$ there is at least one acyclic matching from $A$ to $B$. The main result concerning the acyclic matching property is proven by Losonczy in \cite{8}, where it is shown that any torsion free abelian group has the acyclic matching property. In this section we prove a linear version of this result (see Theorem \ref{th4.5} below).\\
Let $K\subset L$ be a central extension of division rings and let $A$ and $B$ be $n$-dimensional
$K$-subspaces of $L$ for some $n\geq 1$. Then, following Eliahou and Lecouvey in \cite{6}, we say that a linear isomorphism $f : A \to B$ is a strong matching from $A$ to $B$ if every ordered basis $\mathcal{A}$ of $A$ is matched to the basis $\mathcal{B} := f(\mathcal{A})$ of $B$, under the bijection induced by $f$. We will need the following criterion for the existence of a strong matching.
\begin{theorem}\label{th4.1}
\cite[Theorem 6.3]{6} Using the above notations, there is a strong matching from $A$ to $B$ if and only if $AB \cap A = \{0\}$. Moreover, in this case, any linear isomorphism $f : A\to B$
is a strong matching.
\end{theorem}
Now, we want to introduce the acyclicity property for strong matchings. In order to motivate our definition, we make the following observation.
\begin{proposition}\label{pr4.2}
 Let $G$ be a group, $A$ and $B$ be two finite subsets of $G$ with the same cardinality, and $f, g : A \to B$ be two matchings from $A$ to $B$. Then $m_f = m_g$ if and only if there exists a bijection $\phi : A\to A$ such that for all a in $A, af(a) =\phi (a)g(\phi (a))$.
 \end{proposition}
 \begin{proof}
That the existence of a bijection $\phi $ with the stated property implies $m_f = m_g$ is immediate. For the other direction, assume that $m_f = m_g$ and, for any element $x$ in $G$, consider the following subsets of $A$.
\begin{eqnarray*}
&&A^f_x := \{ a \in A : af(a) = x\},\\
&&A^g_x := \{a \in A : ag(a) = x\} .
\end{eqnarray*}
It is clear that $\mathcal{P}^f = \{A^f_x\} _{x\in G}$ and $\mathcal{P}^g = \{A^g_x\}_{x\in G}$ give two partitions of the set $A$ and that, under the assumption $m_f = m_g, |A^f_x| = |A^g_x|$ for all $x\in G$. Now, for any $x \in G$ fix an arbitrary bijection $\phi _x$ from $A^f_x$ to $A^g_x$ and glue all this bijections to get a bijection $\phi $ from $A$ to $A$. This bijection $\phi $ satisfies the condition in the statement of the proposition. 
\end{proof}\\
Back to the linear setting, motivated by the above observation, we say that two linear isomorphisms $f,g : A\to B$ are equivalent if there exists a linear automorphism $\phi : A\to A$ such that for all $a\in A$ one has $af(a) =\phi (a)g(\phi (a))$, and two strong matchings $f, g : A \to  B$ are equivalent if they are equivalent as linear isomorphisms. Then, we define an acyclic matching from $A$ to $B$ to be a strong matching $f : A\to B$ such that for any strong matching
$g : A \to B$ that is equivalent to $f$, one has $f = cg$ for some constant $c\in K$. Finally, we say that the extension $K\subset L$ satisfies the linear acyclic matching property if for every pair $A$ and $B$ of nonzero equi-dimensional $K$-subspaces of $L$ with $AB \cap A =\{0\}$, there is at least one acyclic matching from $A$ to $B$. Now we are going to prove that every purely transcendental field extension $K\subset L$ satisfies the linear acyclic matching property. We start
with the following:
\begin{lemma}\label{l4.3}
 Let $K \subset L$ be a purely transcendental field extension, $A$ and $B$ be two nonzero finite dimensional $K$-subspaces of $L$ with the same dimension, and $f,g : A \to B$ be two equivalent linear isomorphisms from $A$ to $B$. Then either $f = cg$ for a nonzero constant $c \in K$ or $B =\alpha A$ for a nonzero constant $\alpha \in L$. Further, in the latter case, if the equivalence between $f$ and $g$ is given by means of a linear automorphism $\phi : A \to A$, then $g\circ \phi $ is the multiplication by $\alpha $ map.
 \end{lemma}
 \begin{proof}
 Let $\phi : A \to A$ be a linear automorphism such that for all $a \in A, af(a) =\phi (a)g(\phi (a))$. Fix a nonzero element $x$ in $A$ and consider the following subspaces of $A$.
 \begin{eqnarray*}
&&V_x :=\{ a \in A : x\phi (a) = a\phi (x)\} ,\\
&&W_x :=\{ a \in A : xg(\phi (a)) = ag(\phi (x))\} . 
 \end{eqnarray*}
We claim that $A =V_x\cup W_x$, or in other words, that for any $a\in A$ one has:
\begin{equation}\label{eq9}
(x\phi (a)-a\phi (x))(xg(\phi (a))-ag(\phi (x))=0.
\end{equation}
First of all, note that we have $(a+x)f(a+x) =\phi (a+x)g(\phi (a+x))$. Expanding this equation, using linearity of $f, g$, and $\phi $, and using $af(a) =\phi (a)g(\phi (a))$ and $xf(x) =\phi (x)g(\phi (x))$, we get:
\begin{equation}\label{eq10}
af(x) + xf(a) =\phi (a)g(\phi (x)) +\phi (x)g(\phi (a)).
\end{equation}
Then we calculate
\begin{eqnarray*}
&&0 = ax(af(x)- af(x)- xf(a) + xf(a)) \stackrel{(10)}{=}\\
&&ax(af(x) -\phi (a)g(\phi (x))-\phi (x)g(\phi (a)) + xf(a)) =\\
&&a^2xf(x)-ax\phi (a)g(\phi (x))- ax\phi (x)g(\phi (a)) + ax^2f(a) =\\
&&a^2\phi (x)g(\phi (x))- ax\phi (a)g(\phi (x))- ax\phi (x)g(\phi (a)) + x^2\phi (a)g(\phi (a)) =\\
&&(x\phi (a)-a\phi (x))(xg(\phi (a))- ag(\phi (x))).
\end{eqnarray*}
This proves equality \eqref{eq9} and therefore we have shown that $A = V_x\cup W_x$. On the other hand, as a vector space can not be the union of two proper subspaces (see the following remark), it implies that either $A = V_x$ or $A = W_x$. Let us consider each case separately.
\begin{itemize}
\item[$\bullet $]
If $A = V_x$, then for every nonzero element $a$ in $A$ we have $\dfrac{\phi (a)}{a}=\dfrac{\phi (x)}{x}$ . Therefore, $\phi = tId_A$ where $t =\dfrac{\phi (x)}{x}$ . On the other hand, being an eigenvalue of the linear operator$\phi $ acting on a finite dimensional $K$-vector space, $t$ is algebraic over $K$. But we assumed that $L$ is purely transcendental over $K$ and hence we have $t\in K$. Finally, for every nonzero element $a$ in $A$, we have:
\[
af(a) =\phi (a)g(\phi (a)) =tag(ta) =t^2ag(a).
\]
This implies that $f = cg$ for $c = t^2 \in K$.
\item[$\bullet $]
If $A = W_x$, then put$\alpha := \dfrac{g(\phi (x))}{x}\in L$ and note that for every nonzero element $a$ in $A$ we have $g(\phi (a)) =\alpha a$. This implies that:
\[
B = g(A) = g(\phi (A)) =\alpha A.
\]
\end{itemize}
This finishes the proof of lemma.
\end{proof}\\
\begin{remark}\label{re4.4}
 In the proof of the above lemma, we have used the fact that a vector space can not be the union of two proper subspaces. Note that if the base field is an infinite field, it is well known that even the union of finitely many proper subspaces do not cover a vector space. Over finite fields, this is obviously not the case, but it is still true if "finitely many" is replaced by "two", which can be deduced from the following simple counting argument. \\
Let the base field under consideration be a finite field with $q$ elements and let $V$ be an $n$-dimensional vector space over this field. Then, the cardinality of $V$ is $q^n$ and the cardinality
of any proper subspace is at most $q^{n-1}$. Since any pair of subspaces have at least the zero element in common, the union of two proper subspaces has at most $2q^{n-1}-1$ elements, which is strictly less than $q^n$ as $q\leq 2$. This implies that $V$ can not be the union of two proper subspaces.
\end{remark}
Now we are ready to prove the following result.
\begin{theorem}\label{th4.5}
 Let $L$ be a purely transcendental field extension of a field $K$. Then $K\subset  L$ satisfies the linear acyclic matching property.
 \end{theorem}
\begin{proof}
 Let $A$ and $B$ be nonzero $n$-dimensional $K$-subspaces of $L$ with $AB\cap A = \{0\}$. By
Theorem \ref{th4.1} we know that any linear isomorphism from $A$ to $B$ is a strong matching and our goal is to show that at least one of these strong matchings is acyclic.\\
Fix a strong matching $f : A\to B$. If every strong matching equivalent to $f$ is of the form $cf$ for a constant $c\in K$, then $f$ is acyclic by denition and we are done. If not, let $g : A\to B$
be a strong matching equivalent to $f$ which is not of the form $cf$. Then, by Lemma \ref{l4.3}, there is a constant $\alpha \in L$ such that $B =\alpha A$. In this case, consider the multiplication by $\alpha $ map $w_\alpha : A \to B$, which is a strong matching by Theorem \ref{th4.1} We claim that $w_\alpha $ is an acyclic matching. Let $h : A\to B$ be a strong matching equivalent to $w_\alpha $, and let $\psi : A \to A$ be a linear automorphism such that for every $a\in A, ah(a) =\psi (a)w_\alpha (\psi (a))$. Then, by Lemma \ref{l4.3}, either $h =cw_\alpha $ for a nonzero constant $c\in K$ or there is a nonzero constant $\beta \in L$ such that $w_\alpha \circ \psi $ is the multiplication by $\beta $ map. In the latter case,$\psi = (\alpha ^{-1}\beta )Id_A$, and thus $\alpha ^{-1}\beta $ is algebraic over $K$ and hence is in $K$ (since $L$ is purely transcendental over $K$). But then, for every $a$ in $A$, we have:
\[
ah(a) = (\alpha ^{-1}\beta )aw_\alpha (\alpha ^{-1}\beta a)=(\alpha ^{-1}\beta )^2aw_\alpha (a).
\]
This implies that $h =(\alpha ^{-1}\beta )^2w_\alpha $, and so we are done.
\end{proof}\\
\textbf{Acknowledgments}: We would like to thank Professor Noga Alon for his useful suggestions
and comments. We would also like to thank the referee for making several useful comments

\small
\begin{center}

\end{center}

 University of Illinois at Chicago, Department of Mathematics, Stat. \& CS, 851 S Morgan
St, Chicago, \textbf{IL} 60607 

\textit{E-mail address:} \texttt{maliab2@uic.edu}\\

California Institute of Technology (Caltech), Department of Mathematics, Pasadena, CA 91125

\textit{E-mail address:} \texttt{hadian@caltech.edu}\\

Sharif University of Technology, Department of Mathematical Sciences, Tehran, IRAN, P. O. Box 11365-9415

\textit{E-mail address:} \texttt{ajafari@sharif.ir}

\end{document}